\documentclass[12pt,fleqn]{amsart}
\usepackage{amsmath}
\usepackage{amsthm}
\usepackage{amssymb,latexsym}
\usepackage{a4wide}

\usepackage[usenames,dvipsnames]{color}
\usepackage[utf8]{inputenc}
\usepackage[T1]{fontenc}
\usepackage{dsfont}
\usepackage{url}
\usepackage{mathrsfs}
\usepackage[mathscr]{euscript}

\usepackage[all]{xy}
\usepackage{rotating}
\usepackage[Symbol]{upgreek}
\usepackage{anysize}
\marginsize{1in}{1in}{0.8in}{1in}

\newtheorem{theorem}{Theorem}[section]

\newtheorem{corollary}[theorem]{Corollary}

\newcounter{maintheorem}

\newtheorem{mainth}[maintheorem]{Theorem}
\newcounter{maincorollary}

\theoremstyle{definition}

\theoremstyle{remark}

\numberwithin{equation}{section}
\def\sqr#1#2{{\,\vcenter{\vbox{\hrule height.#2pt\hbox{\vrule width.#2pt
height#1pt \kern#1pt\vrule width.#2pt}\hrule height.#2pt}}\,}}

\begin{document}
\title[$(1+)$-complemented, $(1+)$-isomorphic copies of $L_{1}$]{$(1+)$-complemented, $(1+)$-isomorphic copies of $L_{1}$\\ in dual Banach spaces}

\author{Dongyang Chen}
\address{School of Mathematical Sciences\\ Xiamen University,
Xiamen, 361005, China}
\email{cdy@xmu.edu.cn}

\author{Tomasz Kania}
\address{Mathematical Institute\\Czech Academy of Sciences\\\v Zitn\'a 25 \\115 67 Praha 1, Czech Republic\\
and\\  Institute of Mathematics and Computer Science\\ Jagiellonian University\\
{\L}ojasiewicza 6\\ 30-348 Krak\'{o}w, Poland}
\email{kania@math.cas.cz, tomasz.marcin.kania@gmail.com}

\author{Yingbin Ruan}
\address{School of Mathematics and Statistics\\ Fujian Normal University,
Fuzhou, 350007, China}
\email{yingbinruan@sohu.com}

\date{\today}


\begin{abstract}
The present paper contributes to the ongoing programme of quantification of isomorphic Banach space theory focusing on the Hagler--Stegall characterisation of dual spaces containing complemented copies of $L_{1}$. As a corollary, we obtain the following quantitative version of the Hagler--Stegall theorem asserting that for a Banach space $X$ the following statements are equivalent:
\begin{itemize}
    \item $X$ contains almost isometric copies of $(\bigoplus_{n=1}^{\infty}\oplus \ell_{\infty}^{n})_{\ell_1}$,
    \item for all $\varepsilon>0$, $X^{*}$ contains a $(1+\varepsilon)$-complemented, $(1+\varepsilon)$-isomorphic copy of $L_{1}$,
    \item for all $\varepsilon>0$, $X^{*}$ contains a $(1+\varepsilon)$-complemented, $(1+\varepsilon)$-isomorphic copy of $C[0,1]^{*}$.
\end{itemize}  Moreover, if $X$ is separable, one may add the following assertion:
\begin{itemize}
    \item for all $\varepsilon>0$, there exists a $(1+\varepsilon)$-quotient map $T\colon X\rightarrow C(\Delta)$ so that $T^{*}[C(\Delta)^{*}]$ is $(1+\varepsilon)$-complemented in $X^{*}$,
\end{itemize}
where $\Delta$ is the Cantor set.
\end{abstract}

\maketitle

\baselineskip=18pt 	
\section{Introduction}

In 1968, Pe{\l}czy\'{n}ski \cite{P} showed that if a Banach space $X$ contains an isomorphic copy of $\ell_{1}$, then the dual space $X^{*}$ contains an isomorphic copy of $L_{1}$ and proved that the converse holds as well subject to a mild technical condition that was later removed by Hagler \cite{H}. More precisely, the result stated that the isomorphic containment of $\ell_1$ is equivalent to the following assertions: $X^{*}$ contains a subspace isomorphic to $L_{1}$, $X^{*}$ contains a subspace isomorphic to $C[0,1]^{*}$. When $X$ is separable, these are further equivalent to the assertions: $X^{*}$ contains a subspace isomorphic to $l_{1}([0,1])$,
and $C[0,1]$ is a quotient of $X$.

Shortly after, Hagler and Stegall \cite{HS} obtained a `complemented' version of the aforementioned Pe{\l}czy\'{n}ski's classical work:

\noindent {\bf Theorem (Hagler--Stegall).}
Let $X$ be a Banach space. Then the following assertions are equivalent:
\begin{enumerate}
\item\label{th1i} $X$ contains a subspace isomorphic to $(\bigoplus_{n=1}^{\infty} \ell_{\infty}^{n})_{\ell_1}$;
\item\label{th1ii} $X^{*}$ contains a complemented subspace isomorphic to $L_{1}$;
\item\label{th1iii} $X^{*}$ contains a complemented subspace isomorphic to $C[0,1]^{*}$;
\item\label{th1iv} $X^{*}$ contains an infinite set $K$ such that $K$ is equivalent to the usual basis of $\ell_{1}(\Gamma)$ for some $\Gamma$, $[K]$ is complemented in $X^{*}$, and $K$ is dense in itself in the weak* topology on $X^{*}$;
\end{enumerate}

\noindent If, in addition, $X$ is separable, then the assertions \eqref{th1i}--\eqref{th1iv} are equivalent to
\begin{itemize}
\item[(5)] There exists a surjective operator $T\colon X\rightarrow C[0,1]$ such that $T^{*}[C[0,1]^{*}]$ is complemented in $X^{*}$.
\end{itemize}

The purpose of this note is to quantify the Hagler--Stegall Theorem in the spirit of a~large number of recent results on quantitative versions of various theorems on and properties of Banach spaces, such as, quantitative versions of Krein's theorem \cite{FHMZ}, Eberlein--\u{S}mulyan and Gantmacher theorems \cite{AC}, James' compactness theorem \cite{CKS}, weak sequential continuity and the Schur property \cite{KPS,KS1}, the (reciprical) Dunford--Pettis property \cite{KS, KKS}, the Banach--Saks property \cite{BKS}, etc. More broadly speaking, the present paper contributes to the on-going programme of quantification of Banach space theory. 

In the present paper, we quantify the Hagler--Stegall theorem by introducing the following three quantities denoted by lower-case Greek letters and defined as infima of certain sets (when the sets happen to be empty, we use the convention that the corresponding value is $\infty$).

Hereinafter $X$ and $Y$ will stand for Banach spaces; $\mathscr{B}(X,Y)$ is the space of (bounded, linear) operators from $X$ to $Y$. We then introduce the following quantities.

\begin{itemize}
    \item $ \alpha_{Y}(X)= \inf\{\operatorname{d}(Y,Z)\colon Z \text{ is a subspace of X }\}$, where $\operatorname{d}(Y,Z)$ is the Banach--Mazur distance between $Y$ and $Z$.\smallskip 
    
            The quantity $\alpha_{Y}(X)$, being directly related to the Banach--Mazur distance, measures how well $Y$ is from being isomorphically embeddable into $X$. Obviously, $\alpha_{Y}(X)=1$ if and only if $X$ contains almost isometric copies of $Y$, that is, for every $\varepsilon>0$, $X$ contains a~subspace  $(1+\varepsilon)$-isomorphic to $Y$.\smallskip

    \item $\beta_{Y}(X)=\inf\{\|A\|\|B\|\colon A\in \mathscr{B}(X,Y), B\in \mathscr{B}(Y,X)$,  $AB=I_{Y}\}$.\smallskip
    
    The quantity $\beta_{Y}(X)$ measures how well $Y$ is from being isomorphic to a complemented subspace of $X$. It is easy to see that $\beta_{Y}(X)=1$ if and only if for every $\varepsilon>0$, there exists a subspace $M$ of $X$ so that $M$ is $(1+\varepsilon)$-isomorphic to $Y$ and $(1+\varepsilon)$-complemented in $X$.\smallskip

    \item $\theta_{Y}(X)=\inf\{\|A\|\|S\|\colon A\in \mathscr{B}(X, Y), S\in \mathscr{B}(X^{*}, Y^{*}),  SA^{*}=I_{Y^{*}}\}$.\smallskip 
    
    The quantity $\theta_{Y}(X)$ measures how well $Y$ is isomorphic to a quotient of $X$ and its dual $Y^{*}$ is isomorphic to a complemented subspace of $X^{*}$.
    We see that $\theta_{Y}(X)=1$ if and only if, for every $\varepsilon>0$, there exists a $(1+\varepsilon)$-quotient map $T\colon X\rightarrow Y$ so that $T^{*}[Y^{*}]$ is $(1+\varepsilon)$-complemented in $X^{*}$.

    \end{itemize}
A straightforward argument shows that
\begin{equation}\label{9}
    \beta_{Y^{*}}(X^{*})\leqslant \theta_{Y}(X)\leqslant \beta_{Y}(X).
\end{equation}

By using the aforementioned three quantities, we quantify the Hagler--Stegall theorem as follows.

\begin{mainth}\label{Thm:B}
Let $X$ be a Banach space. Then
\[
    \alpha_{(\oplus_{n=1}^{\infty} \ell_{\infty}^{n})_{l_{1}}}(X)=\beta_{C[0,1]^{*}}(X^{*})=\beta_{L_{1}}(X^{*}).
\]
If, in addition, $X$ is separable, then
\[
    \theta_{C(\Delta)}(X)=\beta_{L_{1}}(X^{*}).
\]
\end{mainth}

The following $(1+\varepsilon)$-version of the Hagler--Stegall theorem follows from Theorem \ref{Thm:B}.

\begin{corollary}
Let $X$ be a Banach space. Then the following assertions are equivalent:
\begin{enumerate}
    \item[(1)] $X$ contains almost isometric copies of $(\bigoplus_{n=1}^{\infty}\ell_{\infty}^{n})_{l_{1}}$;
    \item[(2)] $X^{*}$ contains a $(1+\varepsilon)$-complemented subspace that is $(1+\varepsilon)$-isomorphic to $L_{1}$ for every $\varepsilon>0$;
    \item[(3)] $X^{*}$ contains a $(1+\varepsilon)$-complemented subspace that is $(1+\varepsilon)$-isomorphic to $C[0,1]^{*}$ for every $\varepsilon>0$.
\end{enumerate}

\noindent If, in addition, $X$ is separable, then
\begin{itemize}
    \item[(4)] For every $\varepsilon>0$, there exists a $(1+\varepsilon)$-quotient map $T\colon X\rightarrow C(\Delta)$ so that $T^{*}[C(\Delta)^{*}]$ is $(1+\varepsilon)$-complemented in $X^{*}$.
\end{itemize}
\end{corollary}

\section{Preliminaries}

Our notation and terminology are standard and mostly in-line with \cite{AK,LT}. Throughout the paper, all Banach spaces can be considered either real or complex. We work with real scalars but the results can be easily amended to the complex too. By a \emph{subspace} we understand a closed, linear subspace and by an \emph{operator} we understand a bounded, linear map.
If $X$ is a Banach space, we denote by $B_{X}$ the closed unit ball of $X$, by $I_{X}$ the identity operator on $X$ and, for a~subset $K\subseteq X$, $[K]$ for the closed linear span of $K$.
For $\lambda\geqslant 1$, we say that a surjective operator $T\colon X\rightarrow Y$ is a $\lambda$-\textit{quotient map} if $\|T\|\operatorname{co}(T)\leqslant \lambda$.
\emph{Quotient maps} are 1-quotient maps according to the above terminology. A~norm-one surjective operator $T\colon X\rightarrow Y$ is a quotient map if and only if $T$ is a $(1+)$-quotient map, that is, a $(1+\varepsilon)$-quotient map for every $\varepsilon>0$.

The \textit{Banach--Mazur distance} $\operatorname{d}(X,Y)$ between two isomorphic Banach spaces $X$ and $Y$ is defined by $\inf\|T\|\|T^{-1}\|$, where the infimum is taken over all isomorphisms $T$ from $X$ onto $Y$. As defined by Lindenstrauss and Rosenthal \cite{LR}, for $\lambda\geqslant 1$, a Banach space $X$ is said to be a $\mathcal{L}_{1,\lambda}$-\textit{space} whenever for every finite-dimensional subspace $E$ of $X$ there is a finite-dimensional subspace $F$ of $X$ such that $F\supseteq E$ and $\operatorname{d}(F,l_{1}^{\dim F})\leqslant \lambda$. We say that a Banach space $X$ is an $\mathcal{L}_{1,\lambda+}$-\textit{space} if it is an $\mathcal{L}_{1,\lambda+\varepsilon}$-\textit{space} for all $\varepsilon>0$.


Following the notation from \cite{HS}, we denote \[
    \mathcal{F}=\{(n,i)\colon n=0, 1,\ldots,  i = 0, 1, \ldots, 2^{n}-1\}
\]
and, for $(n,i), (m,j)\in \mathcal{F}$ we write $(n,i)\geqslant (m,j)$ whenever
\begin{itemize}
    \item $n\geqslant m$,
    \item $2^{n-m}j\leqslant i\leqslant 2^{n-m}(j+1)-1$.
\end{itemize}
Let $\Delta=\{0,1\}^{\mathbb{N}}$ be the Cantor set endowed with the metric
\[\operatorname{d}((a_{n})_{n=1}^\infty,(b_{n})_{n=1}^\infty)=\sum_{n=1}^\infty\frac{1}{2^{n}}|a_{n}-b_{n}|\quad \big((a_{n})_{n},(b_{n})_{n}\in \Delta\big).\]
By Miljutin's Theorem (\cite[Lemma 4.4.7]{AK}), $C[0,1]$ is isomorphic (but not isometric) to $C(\Delta)$. It is well-known that $C(\Delta)^{*}$ and $C[0,1]^{*}$ are linearly isometric, though.

\section{Proof of Theorem \ref{Thm:B}}

The present section is devoted to the proof of Theorem \ref{Thm:B} and is conveniently split into more digestible parts.

\begin{proof}[Proof of Theorem~\ref{Thm:B}] We split the proof into a number of steps.

\noindent{Step 1.} $\beta_{C(\Delta)^{*}}(X^{*})\leqslant \alpha_{(\bigoplus_{n=1}^{\infty}\ell_{\infty}^{n})_{\ell_1}}(X)$.

Since $Z = (\bigoplus_{n=1}^{\infty} \ell_{\infty}^{2^{n}})_{\ell_1}$ embeds isometrically into $(\bigoplus_{n=1}^{\infty} \ell_{\infty}^{n})_{\ell_1}$, it suffices to prove that $\alpha_{Z}(X)\geqslant \beta_{C(\Delta)^{*}}(X^{*})$. For this, let us fix $c>\alpha_{Z}(X)$. Then there exists a~contractive operator $R\colon Z\rightarrow X$ that is bounded below by $1/c$.

Let us consider a double-indexed family $(\Delta_{n,i})_{n=0, i=0}^{\infty, 2^{n}-1}$ of clopen subsets of the Cantor set such that

\begin{enumerate}
    \item $\Delta_{0,0}=\Delta$, $\Delta_{n,i}=\Delta_{n+1,2i}\cup \Delta_{n+1,2i+1}$ ($(n,i)\in \mathcal{F}$) and $\Delta_{n,i}\cap \Delta_{n,j}=\varnothing$ if $i\neq j$;
    \item the diameter of $\Delta_{n,i}$ is $1/{2^{n}}$ ($0\leqslant i\leqslant 2^{n}-1$).
\end{enumerate}
We set $g_{n,i}=\mathds{1}_{\Delta_{n,i}}$, which is a continuous function,  $[g_{n,i}]_{i=0}^{2^{n}-1}\subseteq [g_{n+1,i}]_{i=0}^{2^{n+1}-1}$, $(g_{n,i})_{i=0}^{2^{n}-1}$ is isometrically equivalent to the unit vector basis of $\ell_{\infty}^{2^{n}}$ ($n\in \mathbb N$) and $\bigcup_{n=0}^{\infty}[g_{n,i}]_{i=0}^{2^{n}-1}$ is dense in $C(\Delta)$. 
We may then define an operator $T\colon Z\rightarrow C(\Delta)$ by the assignment $Te_{n,i}=g_{n,i}$. For each $n$, $T$ is an isometry when restricted to $[e_{n,i}\colon 0\leqslant i\leqslant 2^{n}-1]$. Clearly, $\|T\|=1$.\medskip

\noindent \emph{Claim} 1. If $W$ is a finite-dimensional Banach space and $S\colon W\rightarrow C(\Delta)$ is an operator, then for every $\varepsilon>0$, there exists an operator $\widehat{S}\colon W\rightarrow Z$ so that $\|\widehat{S}\|\leqslant (1+\varepsilon)\|S\|$ and $\|S-T\widehat{S}\|\leqslant \varepsilon$.

\begin{proof}[Proof of Claim 1]Let us fix an Auerbach basis $(w_{k},w_{k}^{*})_{k=1}^{N}$ for $W$ ($\dim W=N$). So if $w=\sum_{k=1}^{N}a_{k}w_{k}\in W$, then for each $1\leqslant j\leqslant N$, we get 
\[ 
    |a_{j}|\leqslant |\langle w^{*}_{j},\sum_{k=1}^{N}a_{k}w_{k}\rangle|\leqslant \|w^{*}_{j}\|\|w\|=\|w\|.
\]
It follows that $\sum_{k=1}^{N}|a_{k}|\leqslant N\|w\|.$
Let $\delta>0$ be such that $\delta N\leqslant \varepsilon \|S\|$ and $\delta N\leqslant \varepsilon$. Then, there exist a $n\in \mathbb N$ and $(f_{k})_{k=1}^{N}$ in $[g_{n,i}]_{i=0}^{2^{n}-1}$ so that $\|Sw_{k}-f_{k}\|<\delta$ ($k=1, 2, \ldots, N$). Let us write $f_{k}=\sum_{i=0}^{2^{n}-1}t_{k,i}g_{n,i}$ $(k=1, 2, \ldots, N).$

Define an operator $\widehat{S}\colon W\rightarrow Z$ by 
\[
    \widehat{S}w_{k}=\sum _{i=0}^{2^{n}-1}t_{k,i}e_{n,i}.
\]
We \emph{claim} that $\|\widehat{S}\|\leqslant (1+\varepsilon)\|S\|$ and $\|S-T\widehat{S}\|\leqslant \varepsilon$. Indeed, for $w=\sum_{k=1}^{N}a_{k}w_{k}\in W$, we have
\[
    \begin{array}{lcl}
        \|\widehat{S}w\| & = & \|\sum _{k=1}^{N}a_{k}\widehat{S}w_{k}\|=\|\sum _{k=1}^{N}a_{k}T\widehat{S}w_{k}\|\\
        &=&\|\sum _{k=1}^{N}a_{k}f_{k}\|
        \leqslant \|\sum _{k=1}^{N}a_{k}(f_{k}-Sw_{k})\|+\|\sum _{k=1}^{N}a_{k}Sw_{k}\| \\
        & \leqslant & \sum _{k=1}^{N}|a_{k}|\|f_{k}-Sw_{k}\|+\|S\|\|w\| \\
        &\leqslant & N\|w\|\delta+\|S\|\|w\|\\
        & \leqslant & (1+\varepsilon)\|S\|\|w\|.
    \end{array}
\]
Furthermore,
\[
    \begin{array}{lcl}
    \|Sw-T\widehat{S}w\| &=&\|\sum _{k=1}^{N}a_{k}(Sw_{k}-\sum _{i=0}^{2^{n}-1}t_{k,i}g_{n,i})\|\\
    &=&\|\sum _{k=1}^{N}a_{k}(Sw_{k}-f_{k})\|\\
    &\leqslant&\delta N\|w\|\\
    &\leqslant&\varepsilon \|w\|.
    \end{array}
\]
\end{proof}

Let $\varepsilon>0$. Since $C(\Delta)$ has the metric approximation property, there exists a net $(T_{\alpha})_{\alpha}$ of finite-rank operators on $C(\Delta)$ such that
\begin{itemize}
    \item $\limsup _{\alpha}\|T_{\alpha}\|\leqslant 1+\varepsilon$,
    \item $\dim T_{\alpha}(C(\Delta))\rightarrow \infty$,
    \item $T_{\alpha}\rightarrow I_{C(\Delta)}$ strongly.
\end{itemize}
For each $\alpha$, we may apply Claim 1 to the inclusion map $I_{\alpha}\colon T_{\alpha}[C(\Delta)]\rightarrow C(\Delta)$ in order to get an~operator $\widehat{I_{\alpha}}\colon T_{\alpha}[C(\Delta)]\rightarrow Z$ such that
\begin{itemize}
    \item $\|\widehat{I_{\alpha}}\|\leqslant 1+\varepsilon$,
    \item $\|I_{\alpha}-T\widehat{I_{\alpha}}\|\leqslant (1+\dim T_{\alpha}[C(\Delta)])^{-2}$.
\end{itemize} Hence, for $f\in C(\Delta)$, we get
\[
\begin{array}{lcl}
    \|T\widehat{I_{\alpha}}T_{\alpha}f-f\|
    &\leqslant& \|T\widehat{I_{\alpha}}T_{\alpha}f-I_{\alpha}T_{\alpha}f\|+\|T_{\alpha}f-f\|\\
    &\leqslant & \|T\widehat{I_{\alpha}}-I_{\alpha}\|\|T_{\alpha}\|\|f\|+\|T_{\alpha}f-f\|\rightarrow 0.
\end{array}
\]
Let $S$ be a $\sigma(\mathscr{B}(Z^{*},C(\Delta)^{*}),Z^{*}\widehat{\otimes}_{\pi}C(\Delta))$-cluster point of the net $((\widehat{I_{\alpha}}T_{\alpha})^{*})_{\alpha}$. We show that $ST^{*}=I_{C(\Delta)^{*}}$. Indeed, we choose a subnet $((\widehat{I_{\alpha'}}T_{\alpha'})^{*})_{\alpha'}$ of $((\widehat{I_{\alpha}}T_{\alpha})^{*})_{\alpha}$ so that
$(\widehat{I_{\alpha'}}T_{\alpha'})^{*}\rightarrow S$ in $\sigma(\mathscr{B}(Z^{*},C(\Delta)^{*}),Z^{*}\widehat{\otimes}_{\pi}C(\Delta))$-topology. Then, for $f\in C(\Delta)$ and $\mu\in C(\Delta)^{*}$, we get $\langle (\widehat{I_{\alpha'}}T_{\alpha'})^{*}T^{*}\mu,f\rangle\rightarrow \langle ST^{*}\mu,f\rangle$. On the other hand, we have $$\langle (\widehat{I_{\alpha'}}T_{\alpha'})^{*}T^{*}\mu,f\rangle=\langle \mu, TI_{\alpha'}T_{\alpha'}f\rangle\rightarrow \langle\mu,f\rangle.$$
Therefore, $\langle ST^{*}\mu,f\rangle=\langle\mu,f\rangle.$

\medskip

\noindent \emph{Claim} 2. There exists an operator $\widetilde{T}\colon C(\Delta)^{*}\rightarrow X^{*}$ so that $R^{*}\widetilde{T}=T^{*}$ and $\|\widetilde{T}\|\leqslant c(1+\varepsilon)$.

The proof of the claim is a variation of the Lindenstrauss' compactness argument (see \cite[Proposition 1]{Johnson:1972} and \cite[Lemma 2]{Lindenstrauss:1966}). Since some amendments are required, we present the full reasoning.

\begin{proof}[Proof of Claim 2] We use the fact that $C(\Delta)^{*}$ is isometric to $L_{1}(\mu)$ for some infinite measure $\mu$, and as such, it is a $\mathcal{L}_{1,1+}$-space. Let $\Lambda$ be the collection of all finite-dimensional subspaces of $C(\Delta)^{*}$. Then, for each $\gamma\in \Lambda$ there exist $E_{\gamma}\in \Lambda$ with $\gamma\subseteq E_{\gamma}$ together with an isomorphism $U_{\gamma}\colon \ell_{1}^{\dim E_{\gamma}}\rightarrow E_{\gamma}$ so that $\|U_{\gamma}\|\|U_{\gamma}^{-1}\|\leqslant 1+\varepsilon$. Let $S_{\gamma}\colon Z\rightarrow E_{\gamma}^{*}$ be an operator such that $S_{\gamma}^{*}=T^{*}|_{E_{\gamma}}$ ($\gamma\in \Lambda$). By the 1-injectivity of $\ell_{\infty}^{\dim E_{\gamma}}$, there is an operator $R_{\gamma}\colon X\rightarrow \ell_{\infty}^{\dim E_{\gamma}}$ so that $R_{\gamma}R=U_{\gamma}^{*}S_{\gamma}$ and $\|R_{\gamma}\|\leqslant \|U_{\gamma}^{*}S_{\gamma}\|\|R^{-1}\|\leqslant \|U_{\gamma}\|\|T\|\|R^{-1}\|$. Let $T_{\gamma}=R_{\gamma}^{*}U_{\gamma}^{-1}\colon E_{\gamma}\rightarrow X^{*}$. Then $R^{*}T_{\gamma}= T^{*}|_{E_{\gamma}}$ and $\|T_{\gamma}\|\leqslant c(1+\varepsilon)\|T\|$. For each $\gamma$, we define a non-linear, discontinuous function from $C(\Delta)^{*}$ to $X^{*}$ by
\[
    \widetilde{T_{\gamma}}f=
        \left\{ \begin{array}
                    {l@{\quad}l}

            T_{\gamma}f, & f\in E_{\gamma}\\ 0, & \text{otherwise.}
        \end{array} \right.
\]
Then $(\widetilde{T_{\gamma}})_{\gamma}$ is a net in the compact space
\[
    \prod _{f\in C(\Delta)^{*}}c(1+\varepsilon)\|T\|\|f\|B_{X^{*}}.
\]
and as such, it has a cluster point $\widetilde{T}$. Standard arguments show that $\widetilde{T}$ is linear, $R^{*}\widetilde{T}=T^{*}$ and $\|\widetilde{T}\|\leqslant c(1+\varepsilon)\|T\|=c(1+\varepsilon)$.\end{proof}

Finally, we get $SR^{*}\widetilde{T}=ST^{*}=I_{C(\Delta)^{*}}$ and hence
$$\beta_{C(\Delta)^{*}}(X^{*})\leqslant \|\widetilde{T}\|\|SR^{*}\|\leqslant c(1+\varepsilon)^{3}.$$

Letting $\varepsilon\rightarrow 0$, we get $\beta_{C(\Delta)^{*}}(X^{*})\leqslant c$. As $c$ is arbitrary, we get Step 1.

\noindent{Step 2.} $\beta_{L_{1}}(X^{*})\leqslant \beta_{C[0,1]^{*}}(X^{*})$.

It is well known that $L_1$ is isometric to a $1$-complemented subspace of $C[0,1]^{*}$ (see, \emph{e.g.}, \cite[p.~85]{AK}), which implies Step 2.\bigskip

\noindent{Step 3.} $\alpha_{(\bigoplus_{n=1}^{\infty} \ell_{\infty}^{n})_{\ell_1}}(X)\leqslant \beta_{L_{1}}(X^{*})$.

Let $c>\beta_{L_{1}}(X^{*})$. Then there exist operators $A\colon L_{1}\rightarrow X^{*}, B\colon X^{*}\rightarrow L_{1}$ so that $BA=I_{L_{1}}, \|A\|=1$ and $\|B\|<c$. Let $0<\varepsilon<1$ and $\varepsilon_{n}=\varepsilon/2^{2n+3}$ ($n=0,1,\ldots$).

By \cite[Lemma 3]{HS}, we get $(f_{n,i})_{(n,i)\in \mathcal{F}}$ in $L_{\infty}$ and $(x_{n,i})_{(n,i)\in \mathcal{F}}$ in $X$ satisfying
\begin{enumerate}
    \item $\|f_{n,i}\|_{1}=1$ and $f_{n,i}\geqslant 0$ everywhere for all $(n,i)\in \mathcal{F}$;
    \item  For each $n$ and $i\neq j$, $f_{n,i}(t)$ and $f_{n,j}(t)$ cannot be both non-zero for the same $t\in [0,1]$;
    \item\label{B3} \[\langle Af_{n,i},x_{m,j}\rangle= \left\{ \begin{array}
                    {r@{\quad}l}

 1, & (n,i)\geqslant (m,j),\\ 0, & \text{otherwise;}
 \end{array} \right. \]
 \item\label{B4} $\max_{0\leqslant i\leqslant 2^{n}-1}|t_{i}|\leqslant \|\sum _{i=0}^{2^{n}-1}t_{i}x_{n,i}\|\leqslant c(1+\varepsilon_{n}) \max _{0\leqslant i\leqslant 2^{n}-1}|t_{i}|$\\ ($n=0,1,\cdots$; $t_{0},\ldots, t_{2^{n}-1}\in \mathbb R$).
\end{enumerate}

We may now define recursively a sequence $(W_{n,i})_{(n,i)\in \mathcal{F}}$ of non-empty weak*-closed subsets of $B_{X^{*}}$ as follows:

\begin{itemize}
    \item $W_{0,0}=\{x^{*}\in B_{X^{*}}\colon |\langle x^{*},x_{0,0}\rangle-1|\leqslant \varepsilon_{0}\},$
    \item $
        W_{1,0}=W_{0,0}\cap \{x^{*}\in B_{X^{*}}\colon |\langle x^{*},x_{1,0}\rangle-1|\leqslant \varepsilon_{1}, |\langle x^{*},x_{1,1}\rangle|\leqslant \varepsilon_{1}\},
       $
      \item $W_{1,1}=W_{0,0}\cap \{x^{*}\in B_{X^{*}}\colon |\langle x^{*},x_{1,1}\rangle-1|\leqslant \varepsilon_{1}, |\langle x^{*},x_{1,0}\rangle|\leqslant \varepsilon_{1}\},$
      \item $W_{2,0}=W_{1,0}\cap \{x^{*}\in B_{X^{*}}\colon |\langle x^{*},x_{2,0}\rangle-1|\leqslant \varepsilon_{2}, |\langle x^{*},x_{2,j}\rangle|\leqslant \varepsilon_{2},j=1,2,3\},$
      \item $W_{2,1}=W_{1,0}\cap \{x^{*}\in B_{X^{*}}\colon |\langle x^{*},x_{2,1}\rangle-1|\leqslant \varepsilon_{2}, |\langle x^{*},x_{2,j}\rangle|\leqslant \varepsilon_{2},j=0,2,3\},$
      \item $W_{2,2}=W_{1,1}\cap \{x^{*}\in B_{X^{*}}\colon |\langle x^{*},x_{2,2}\rangle-1|\leqslant \varepsilon_{2}, |\langle x^{*},x_{2,j}\rangle|\leqslant \varepsilon_{2},j=0,1,3\},$
      \item $W_{2,3}=W_{1,1}\cap \{x^{*}\in B_{X^{*}}\colon|\langle x^{*},x_{2,3}\rangle-1|\leqslant \varepsilon_{2}, |\langle x^{*},x_{2,j}\rangle|\leqslant \varepsilon_{2},j=0,1,2\}$,
\end{itemize}
and so on. By \eqref{B3}, each $W_{n,i}$ is non-empty. By the choice of $\varepsilon_{n}$, the sets $W_{n,i},  W_{n,j}$ are disjoint as long as $i\neq j$. Let
\[
    K=\bigcap _{n=0}^{\infty}(\bigcup _{i=0}^{2^{n}-1}W_{n,i})\quad\text{and}\quad K_{n,i}=W_{n,i}\cap K\;\; \big((n,i)\in \mathcal{F}\big).\]
    By \eqref{B3}, $Af_{n,i}\in W_{m,j}$ if $(n,i)\geqslant (m,j)$, which implies that each $K_{n,i}$ is non-empty. By the construction of the sequence $(W_{n,i})$, we see that $K_{0,0}=K, K_{n+1,2i}\cup K_{n+1,2i+1}=K_{n,i}$ and $K_{n,i}\cap K_{n,j}=\varnothing$ if $i\neq j$.

    Let us define an operator $T\colon X\rightarrow C(K)$ by $\langle Tx,x^{*}\rangle=\langle x^{*},x\rangle$ ($x\in X,x^{*}\in K$). Then $|\langle Tx_{n,i},x^{*}\rangle-1|\leqslant \varepsilon_{n}$ if $x^{*}\in K_{n,i}$, and $|\langle Tx_{n,i},x^{*}\rangle|\leqslant \varepsilon_{n}$ if $x^{*}\in \bigcup_{j\neq i}K_{n,j}$. Set $g_{n,i}=\mathds{1}_{K_{n,i}}$, which is continuous as $K_{n,i}$ is clopen. Then $\|Tx_{n,i}-g_{n,i}\|\leqslant \varepsilon_{n}$. Moreover, $[g_{n,i}]_{i=0}^{2^{n}-1}\subseteq [g_{n+1,i}]_{i=0}^{2^{n+1}-1}$, $(g_{n,i})_{i=0}^{2^{n}-1}$ is isometrically equivalent to the unit vector basis of $\ell_{\infty}^{2^{n}}$ for all $n$, and
    \[
        [g_{n,i}\colon (n,i)\in \mathcal{F}]=\overline{\bigcup _{n=0}^{\infty}[g_{n,i}]_{i=0}^{2^{n}-1}}
    \]
is isometric to $C(\Delta)$. Let $Z$ be a subspace of $C(\Delta)$ isometric to $(\bigoplus_{n=1}^{\infty} \ell_{\infty}^{n})_{\ell_1}$ and let $(z_{n,j})_{n=1,j=0}^{\infty,n-1}$ be a basis of $Z$ isometrically equivalent to the unit vector basis of $(\bigoplus_{n=1}^{\infty} \ell_{\infty}^{n})_{\ell_1}$. Fix  $n\geqslant 1$. Then there exist $m>n$ and unit vectors $h_{n,j}\in [g_{m,i}]_{i=0}^{2^{m}-1}$ so that $\|z_{n,j}-h_{n,j}\|\leqslant \varepsilon/2^{n+3}$ ($j=0,1,\ldots,n-1$). We write $h_{n,j}=\sum _{i=0}^{2^{m}-1}a_{i,j}g_{m,i}$ and define $y_{n,j}=\sum _{i=0}^{2^{m}-1}a_{i,j}x_{m,i}\in X.$

\noindent \emph{Claim 2a.} For all $(t_{n,j})_{n=1,j=0}^{\infty, n-1}\in (\bigoplus_{n=1}^{\infty} \ell_{\infty}^{n})_{\ell_1}$ we have
\[
    (1-\frac{\varepsilon}{2})\sum _{n=1}^{\infty}\max _{0\leqslant j\leqslant n-1}|t_{n,j}|\leqslant \|\sum _{n=1}^{\infty}\sum _{j=0}^{n-1}t_{n,j}y_{n,j}\|\leqslant c(1+\varepsilon)^{2}\sum _{n=1}^{\infty}\max _{0\leqslant j\leqslant n-1}|t_{n,j}|.
\]

Indeed, by \eqref{B4} we get
\begin{align*}
\|\sum _{j=0}^{n-1}t_{n,j}y_{n,j}\|&=\|\sum _{i=0}^{2^{m}-1}(\sum _{j=0}^{n-1}a_{i,j}t_{n,j})x_{m,i}\|\\
&\leqslant c(1+\varepsilon_{m})\max_{0\leqslant i\leqslant 2^{m}-1}|\sum _{j=0}^{n-1}a_{i,j}t_{n,j}|\\
&=c(1+\varepsilon_{m})\|\sum _{j=0}^{n-1}t_{n,j}h_{n,j}\|\\
&\leqslant c(1+\varepsilon_{m})\Big(\|\sum _{j=0}^{n-1}t_{n,j}z_{n,j}\|+\sum _{j=0}^{n-1}t_{n,j}(h_{n,j}-z_{n,j})\|\Big)\\
&\leqslant c(1+\varepsilon_{m})\Big(\max _{0\leqslant j\leqslant n-1}|t_{n,j}|+n\varepsilon/2^{n+3}\max _{0\leqslant j\leqslant n-1}|t_{n,j}|\Big)\\
&\leqslant c(1+\varepsilon)^{2}\max _{0\leqslant j\leqslant n-1}|t_{n,j}|.
\end{align*}
Consequently, $$\|\sum _{n=1}^{\infty}\sum _{j=0}^{n-1}t_{n,j}y_{n,j}\|\leqslant \sum _{n=1}^{\infty}\|\sum _{j=0}^{n-1}t_{n,j}y_{n,j}\|\leqslant c(1+\varepsilon)^{2}\sum _{n=1}^{\infty}\max _{0\leqslant j\leqslant n-1}|t_{n,j}|.$$
On the other hand, by the choice of $m$ and $h_{n,j}$, we arrive at
\begin{align*}
\|Ty_{n,j}-z_{n,j}\|&\leqslant \|Ty_{n,j}-h_{n,j}\|+\|h_{n,j}-z_{n,j}\|\\
&=\|\sum _{i=0}^{2^{m}-1}a_{i,j}(Tx_{m,i}-g_{m,i})\|+\varepsilon/2^{n+3}\\
&\leqslant \varepsilon_{m}2^{m}\max _{0\leqslant i\leqslant 2^{m}-1}|a_{i,j}|+\varepsilon/2^{n+3}\\
&\leqslant \varepsilon/2^{n+3}+\varepsilon/2^{n+3}=\varepsilon/2^{n+2}.
\end{align*}
This implies
\begin{align*}
\|\sum _{n=1}^{\infty}\sum _{j=0}^{n-1}t_{n,j}y_{n,j}\|&\geqslant \|\sum _{n=1}^{\infty}\sum _{j=0}^{n-1}t_{n,j}Ty_{n,j}\|\\
&\geqslant \|\sum _{n=1}^{\infty}\sum _{j=0}^{n-1}t_{n,j}z_{n,j}\|-\|\sum _{n=1}^{\infty}\sum _{j=0}^{n-1}t_{n,j}(Ty_{n,j}-z_{n,j})\|\\
&\geqslant \sum _{n=1}^{\infty}\max _{0\leqslant j\leqslant n-1}|t_{n,j}|-\sum _{n=1}^{\infty}n\max _{0\leqslant j\leqslant n-1}|t_{n,j}|\frac{\varepsilon}{2^{n+2}}\\
&\geqslant (1-\frac{\varepsilon}{2})\sum _{n=1}^{\infty}\max _{0\leqslant j\leqslant n-1}|t_{n,j}|.\\
\end{align*}
Finally, by Claim, we get
\[
    \alpha_{(\bigoplus_{n=1}^{\infty} \ell_{\infty}^{n})_{\ell_1}}(X)\leqslant c(1+\varepsilon)^{2}/(1-\frac{\varepsilon}{2}).
\]
Letting $\varepsilon\rightarrow 0$ yields $\alpha_{(\bigoplus_{n=1}^{\infty} \ell_{\infty}^{n})_{\ell_1}}(X)\leqslant c$; since $c$ was arbitrary the proof of Step 3 is complete.\medskip

\noindent{Step 4.} $\beta_{L_{1}}(X^{*})\leqslant \theta_{C(\Delta)}(X)$.

This step follows from \eqref{9} together with Step 2. We are now ready to establish the final step of the proof.\medskip

\noindent{Step 5.} Suppose that $X$ is separable. Then $\theta_{C(\Delta)}(X)\leqslant \beta_{L_{1}}(X^{*})$.

Let $c>\beta_{L_{1}}(X^{*})$. Then there exist operators $A\colon L_{1}\rightarrow X^{*}, B\colon X^{*}\rightarrow L_{1}$ so that $BA=I_{L_{1}}, \|A\|=1$, and $\|B\|<c$.

Let $(f_{n,i})_{(n,i)\in \mathcal{F}}$ be a family of functions in $L_{\infty}$, $(x_{n,i})_{(n,i)\in \mathcal{F}}$ in $X$, and $(W_{n,i})_{(n,i)\in \mathcal{F}}$  associated to $\varepsilon_{n}=1/2^{2n+2}$ ($n=0,1,\ldots$) as described in Step 3. Since $X$ is separable, we may assume that the $\operatorname{d}$-diameter of $W_{n,i}\leqslant 2^{-n}$ for each $i$, where $\operatorname{d}$ is a metric giving the relative $\sigma(X^{*},X)$-topology on $B_{X^{*}}$.
Let
\[
    K=\bigcap _{n=0}^{\infty}(\bigcup _{i=0}^{2^{n}-1}W_{n,i})\quad\text{ and }\quad K_{n,i}=W_{n,i}\cap K\;\; \big((n,i)\in \mathcal{F}\big).
\]
Then $K$ is a compact, totally disconnected metric space without isolated points, hence homeomorphic to $\Delta$. Moreover, $K_{0,0}=K, K_{n+1,2i}\cup K_{n+1,2i+1}=K_{n,i}$ and $K_{n,i}\cap K_{n,j}=\varnothing$ if $i\neq j$. Hence $K=\bigcup_{i=0}^{2^{n}-1}K_{n,i}$ for all $n$. As seen in Step 3, the operator
$T\colon X\rightarrow C(K)$, defined by $\langle Tx,x^{*}\rangle=\langle x^{*},x\rangle$ ($x\in X,x^{*}\in K$), satisfies $\|Tx_{n,i}-g_{n,i}\|\leqslant \varepsilon_{n}$, where
$g_{n,i}=\mathds{1}_{K_{n,i}}\in C(K)$.

An argument analogous to Step 1 yields that, if $W$ is a finite-dimensional Banach space and $S\colon W\rightarrow C(K)$ is an operator, then, for every $\varepsilon>0$, there exists an operator $\widehat{S}\colon W\rightarrow X$ so that $\|\widehat{S}\|\leqslant c(1+\varepsilon)\|S\|$ and $\|S-T\widehat{S}\|\leqslant \varepsilon$.

Fix $\varepsilon>0$. By an argument analogous to the one from Step 1, we get an operator $S\colon X^{*}\rightarrow C(K)^{*}$ with $\|S\|\leqslant c(1+\varepsilon)^{2}$ so that $ST^{*}=I_{C(K)^{*}}$. This means that \[
    \theta_{C(\Delta)}(X)=\theta_{C(K)}(X)\leqslant c(1+\varepsilon)^{2}.
\]
Letting $\varepsilon\rightarrow 0$, we arrive at $\theta_{C(\Delta)}(X)\leqslant c.$
As $c$ is arbitrary, the proof is complete.\end{proof}

\subsection*{Acknowledgements}
The authors would like to thank Prof. W. B. Johnson for helpful discussions and comments. The first-named author was supported by the National Natural Science Foundation of China (Grant No.11971403) and the Natural Science Foundation of Fujian Province of China (Grant No. 2019J01024). The second-named author acknowledges with thanks funding received from SONATA 15 No.~2019/35/D/ST1/01734.


\begin{thebibliography}{MMM}
\frenchspacing

\bibitem{AK}
F. Albiac and N.~J. Kalton,
{\em Topics in Banach space theory}, Springer, 2005.


\bibitem{AC}
C. Angosto and B. Cascales, {\em Measures of weak non-compactness in Bananch spaces}, Topology Appl. {\bf 156}(2009), 1412-1421.


\bibitem{BKS}
H. Bendov\'{a}, O. F. K. Kalenda and J. Spurn\'{y}, {\em Quantification of the Banach-Saks property}, J. Funct. Anal. {\bf 268}(2015), 1733-1754.


\bibitem{Ca} P. G. Casazza, {\em Approximation properties}, \textrm{Handbook of the geometry of Banach spaces, Vol.1},
W.~B. Johnson and J. Lindenstrauss, eds, Elsevier, Amsterdam (2001), 271--316.


\bibitem{CKS}
B. Cascales, O. F. K. Kalenda and J. Spurn\'{y}, {\em A quantitative version of James' compactness theorem}, Proc. Edinburgh Math. Soc. {\bf 55}(2012), 369-386.


\bibitem{FHMZ}
M. Fabian, P. H\'{a}jek, V. Montesinos and V. Zizler, {\em A quantitative version of Krein's theorem}, Rev. Mat. Iberoamer. {\bf 21}(2005), 237-248.


\bibitem{H}
J. Hagler, {\em Some more Banach spaces which contain $l^{1}$}, Studia Math. {\bf 46} (1973), 35--42.


\bibitem{HS}
J. Hagler and C. Stegall, {\em Banach spaces whose duals contain complemented subspaces isomorphic to $C[0,1]^{*}$}, J. Funct. Anal. {\bf 13} (1973), 233--251.

\bibitem{Johnson:1972} W. B. Johnson, \emph{A complementary universal conjugate Banach space and its relation to the approximation problem}, Israel J. Math. \textbf{13} (3-4) (1972), 301--310.

\bibitem{KKS}
M. Ka\v{c}ena, O. F. K. Kalenda, and J. Spurn\'{y}, {\em Quantitative Dunford--Pettis property}, Adv. Math. {\bf 234} (2013), 488--527.


\bibitem{KPS}
O.~F.~K.~Kalenda, H.~Pfitzner, and J.~Spurn\'{y}, {\em On quantification of weak sequential completeness}, J. Funct. Anal. {\bf 260} (2011), 2986--2996.


\bibitem{KS1}
O. F. K. Kalenda and J. Spurn\'{y}, {\em On a difference between quantitative weak sequential completeness and the quantitative Schur property}, Proc. Amer. Math. Soc. {\bf 140}(2012), 3435-3444.


\bibitem{KS}
O. F. K. Kalenda and J. Spurn\'{y}, {\em Quantification of the reciprocal Dunford--Pettis property}, Studia Math. {\bf 210}(2012), 261-278.


\bibitem{Lindenstrauss:1966}  J.~Lindenstrauss, \emph{On nonseparable reflexive Banach spaces}, Bull. Amer. Math. Soc. {\bf 72} (1966), 967--970.

\bibitem{LR} J. Lindenstrauss and H. P. Rosenthal,
{\em The $\mathcal{L}_{p}$ spaces}, Israel J. Math. {\bf 7} (1969), 325--349.

\bibitem{LT} J. Lindenstrauss and L. Tzafriri,
{\em Classical Banach Spaces I, Sequence Spaces}, Springer, Berlin, 1977.


\bibitem{P}
A. Pe{\l}czy\'{n}ski, {\em On Banach spaces containing $L_{1}(\mu)$}, Studia Math. {\bf 30} (1968), 231--246.


\bibitem{S}
C. Stegall, {\em Banach spaces whose duals contain $l_{1}(\Gamma)$ with applications to the study of dual $L_{1}(\mu)$ spaces}, Trans. Amer. Math. Soc. {\bf 176} (1973), 463--477.
\end{thebibliography}
\end{document}